\documentclass[a4paper, 11pt, final,underruledhead,leqno]{paper}
\usepackage{amssymb}
\usepackage{amsmath}
\usepackage{latexsym}
\usepackage{math}
\usepackage[mathscr]{euscript}
\usepackage{enumerate}

\title[Grassmann spaces over symplectic copolar spaces]{%
Projective symplectic geometry on regular subspaces; %
Grassmann spaces over symplectic copolar spaces}
\author{M. Pra{\.z}mowska, K. Pra{\.z}mowski, M. {\.Z}ynel}

\def\dzieli{\mathrel{|}}
\def\ndzieli{\mathrel{\not|}}
\def\LineOn(#1,#2){\overline{{#1},{#2}\rule{0em}{1,5ex}}}

\def\collin{\sim}
\def\wspolin{\mbox{\boldmath$L$}}

\def\oadjac{\mathrel{\mbox{$\collin\mkern-18mu{\raise0.8ex\hbox{$\perp$}}$}}}

\def\agen#1#2{{\mbox{{\boldmath$[\mkern-4.5mu|$}}{#1}\mbox{{\boldmath$|\mkern-4.5mu]$}}_{{#2}}}}

\def\adjac{\mathrel{\sim}}
\def\nadjac{\mathrel{\not\adjac}}
\def\toadjac{\mathrel{\lower.35ex\hbox{\baselineskip-3pt\lineskip-3pt\vbox{\hbox{$\sim$}\hbox{$\sim$}}}}}
\let\doadjac\toadjac
\def\soadjac{\lower.25ex\hbox{\scriptsize\baselineskip-2.2pt\lineskip-2.2pt\vbox{\hbox{$\sim$}\hbox{$\sim$}}}}
\def\ssoadjac{\lower.21ex\hbox{\tiny\baselineskip-2.2pt\lineskip-2.2pt\vbox{\hbox{$\sim$}\hbox{$\sim$}}}}
\def\badjac{\mathchoice{\doadjac}{\toadjac}{\soadjac}{\ssoadjac}}

\def\aftoadjac{\mathrel{\lower.35ex\hbox{\baselineskip-3pt\lineskip-3pt\vbox{\hbox{$\sim$}\hbox{$\simeq$}}}}}
\let\afdoadjac\aftoadjac
\def\afsoadjac{\lower.25ex\hbox{\scriptsize\baselineskip-2.2pt\lineskip-2.2pt\vbox{\hbox{$\sim$}\hbox{$\simeq$}}}}
\def\afssoadjac{\lower.21ex\hbox{\tiny\baselineskip-2.2pt\lineskip-2.2pt\vbox{\hbox{$\sim$}\hbox{$\simeq$}}}}
\def\afbadjac{\mathchoice{\afdoadjac}{\aftoadjac}{\afsoadjac}{\afssoadjac}}

\def\blisk{\mathrel{\wedge}}
\def\bliskx{\mathrel{\hbox{\baselineskip-5pt\lineskip-3pt\vbox{\hbox{\footnotesize$\blisk$}\hbox{\footnotesize$\blisk$}}}}}

\def\blisq{\mathrel{\hbox{\baselineskip-5pt\lineskip-3pt\vbox{\hbox{\footnotesize$\barwedge$}\hbox{\footnotesize$\blisk$}}}}}

\def\penc{{\bf p}}
\def\benc{{\bf b}}

\def\peki{{\cal P}}

\def\afinmark{\propto}
\def\apeki{{\peki}^{\afinmark}}
\def\pek(#1,#2){\penc({#1},{#2})}
\def\bek(#1,#2){\benc({#1},{#2})}
\def\pekx(#1,#2){\penc^{\ast}({#1},{#2})}
\def\apek(#1,#2){\penc^{\afinmark}({#1},{#2})}
\def\abek(#1,#2){\benc^{\afinmark}({#1},{#2})}
\def\apekx(#1,#2){\penc^{\ast\afinmark}({#1},{#2})}
\def\rpek(#1,#2){\penc^{\goth r}({#1},{#2})}

\let\topadjac\upadjac
\let\botadjac\downadjac

\def\stars{{\cal S}}

\def\tops{{\cal T}}
\def\topsx{{\cal T}^\ast}

\def\starof{\mathrm{S}}
\def\topof{\mathrm{T}}
\def\astarof{\starof^{\afinmark}}
\def\atopof{\topof^{\afinmark}}

\def\Quadr{{\mathbf{Q}}}
\def\regsuby{{\mathbf{R}}}
\def\tangsuby{{\mathbf{T}}}
\def\symsuby{{(\mathbf{T\text{-}R})}}

\def\AffineSpSymb{\mathbf{A}}
\def\AfSpace(#1){\ensuremath{\AffineSpSymb(#1)}}

\def\KwadrSpSymb{\mathbf{Q}}
\def\KwadrSpace(#1,#2){\ensuremath{\KwadrSpSymb_{#1}(#2)}}
\def\AfPolSpSymb{\mathbf{U}}
\def\AfpolSpace(#1,#2){\ensuremath{\AfPolSpSymb_{#1}(#2)}}
\def\AfpolSpacex(#1,#2){\ensuremath{\AfPolSpSymb^\dagger_{#1}(#2)}}

\def\PencilSp(#1,#2){%
\def\tempa{#2}
\def\tempb{\symsuby}
\ifx\tempa\tempb
{\bf P}_{#1}{#2}
\else
{\bf P}_{#1}({#2})
\fi}
\def\PencilSpx(#1,#2){{\bf P}_{#1}^\dagger({#2})}
\def\PencSpace(#1,#2){\PencilSp({#2},{#1})}
\def\PencSpacex(#1,#2){\PencilSpx({#2},{#1})}

\def\Rad{{\mathrm{Rad}}}

\def\rdim{{\mathrm{rdim}}}

\def\fixproj{\ensuremath{\goth P}}
\def\fixprojr{\ensuremath{\goth R}}

\newenvironment{ctext}{%
  \par
  \smallskip
  \centering
}{%
 \par
 \smallskip
 \csname @endpetrue\endcsname
}

\newcounter{sentencex}
\def\thesentencex{\Alph{sentencex}}
\def\labelsentencex{\upshape(\thesentencex)}

\newenvironment{sentencesx}{%
   \list{\labelsentencex}
      {\usecounter{sentencex}\def\makelabel##1{\hss\llap{##1}}
        \topsep3pt\leftmargin0pt\itemindent40pt\labelsep8pt}%
  }{%
    \endlist}

\newenvironment{efekt}{%
\refstepcounter{equation}
  \list{{(\theequation)}}
      {\topsep3pt\settowidth{\labelwidth}{(\theequation)}%
         \labelsep\leftmargin\addtolength{\labelsep}{-\labelwidth}%
           \itemindent0pt}%
  \item\em  
}{%
  \endlist}

\def\efekta#1{%
\begin{equation}\begin{minipage}[m]{0.9\textwidth}\em #1\end{minipage}\end{equation}}

\newenvironment{uwaga}[1]{\color{blue}{\sc #1}:\/\color{red}}{\color{black}}

\begin{document}

\maketitle

\begin{abstract}
  We construct Grassmann spaces associated with the incidence
  geometry of regular and tangential subspaces of a 
  symplectic copolar space, 
  show that the underlying metric projective
  space can be recovered in terms of the corresponding adjacencies
  on so distinguished family of $k$-subspaces 
  ($2k+1 \neq$ geometrical dimension of the space),
  and thus we prove that bijections which preserve the adjacency are
  determined by automorphisms of the underlying space.

\end{abstract}

\section*{Introduction}

It is a classical result that the projective geometry can be recovered from
its associated Grassmann space (cf. \cite{bichtalin}, \cite{talin}, \cite{pamb3})
and adjacency preserving bijections of projective $k$-subspaces 
($2k+1 \neq$ the dimension of the space) 
are determined by collineations (cf. e.g. \cite{diedon}).
How to transfer these notions to metric projective geometry
to obtain reasonable results?

Several distinct ways lead to problems investigated in this paper.
Let us consider a metric projective geometry i.e. a projective
space \fixproj \ equipped with a nondegenerate polarity $\varkappa$.
In a more specific sense, let us consider a vector space $\field V$
(which represents \fixproj) equipped with a nondegenerate bilinear
reflexive form $\xi$ (which determines $\varkappa$ and the conjugacy $\perp$).
We also assume that the coordinate field of $\field V$ has characteristic
$\neq 2$.
If $\ind(\xi) > 1$ then the geometry of $(\fixproj,\perp)$
can be expressed in terms of the associated polar space,
whose points are the isotropic (singular) points of \fixproj \ 
and whose lines are
the isotropic lines (cf. \cite{veldkamp}, \cite{bushult}). 
This geometry can be also expressed in terms of
the adjacency: binary collinearity of points.
This result can be extended to isotropic subspaces of higher dimensions, and on 
isotropic subspaces Grassmann spaces can be defined quite naturally (cf. \cite{polargras}).
Analogous results remain valid for affine polar spaces
(``polar spaces" associated with metric affine geometries, 
cf. \cite{cohenshult}, \cite{afpolar}).
But within $(\fixproj,\perp)$ isotropic subspaces are ``most" degenerate.
``Least" degenerate and ``opposite" to  isotropic are 
regular (radical-free) subspaces, which are more
suitable to develop the geometry, especially when reflections are considered
(cf. \cite{reflgeom1}).

If $\ind(\xi)\neq 0$ and $\xi$ is symmetric then the structure of regular points
and regular lines is equivalent to the underlying metric projective space.
One can also extend this result to regular subspaces of higher dimensions
defining associated Grassmann spaces of regulars subspaces.
The underlying metric projective geometry can be expressed in terms of
such Grassmannians (cf. \cite{grasregul}).
However, if $\xi$ is 
symplectic (anti-symmetric or skew-symmetric in other words)
then no point is regular and the above, so elegant approach fails.
Instead, one can consider symplectic copolar space  
(cf. \cite{copolar}, \cite{embcopol}, 
also called hyperbolic symplectic space, cf. \cite{gramlich})
with the isotropic points and the regular lines.
How to extend this approach to higher dimensions?

Secondly, if $\xi$ is symmetric then quite interesting geometry arises
when we consider the structure of regular points and so called tangential lines
(lines, which contain exactly one nonregular point, cf. \cite{tangi}).
If $\xi$ is symplectic, then such a structure has no sense, but instead,
we can consider the structure with regular lines and tangential planes.
Here, a tangential subspace is defined as a subspace, whose radical is a point.

All of that suggests that extending all the machinery of adjacency and 
Grassmann spaces to symplectic copolar spaces one should investigate 
``best possible": regular subspaces of even (linearly computed) dimensions
and tangential subspaces of odd dimensions.
And indeed, as we prove in this paper, classical results concerning geometry
of Grassmannians remain valid: the underlying metric projective geometry
can be expressed in terms of Grassmann spaces of such a family of 
subspaces, and an analogue of Chow Theorem (cf. \cite{chow}, \cite{diedon}) 
holds i.e.
bijections which preserve the adjacency are determined by
automorphisms of the underlying metric projective space.

\section{Notions, results}\label{sec:notions}

Let $\field V$ be a vector space with the field of scalars of $\mathrm{char}\neq 2$,
let a nondegenerate symplectic form $\xi$ be defined on $\field V$,
and let $n = \dim({\field V})$.
Then $n = 2m$, where $m = \ind(\xi)$.

Let $\Sub({\field V})$ ($\Sub_k({\field V})$, resp.) stand for 
all the (all the $k$-dimensional) subspaces of $\field V$.
For any ${\mathscr H} \subset \Sub({\field V})$ and integer $k$ we write
${\mathscr H}_k = {\mathscr H}\cap \Sub_{k}({\field V})$.

\par\noindent
The structure
\begin{ctext}
  $\fixproj := \PencSpace({\field V},1) := 
  \struct{\Sub_1({\field V}),\Sub_2({\field V}),\subset}$
\end{ctext}
is the projective space over $\field V$.
In what follows we shall refer mostly to linear dimension, 
so a point of \fixproj \ has dimension $1$,
a line has dimension $2$ and so on.
The orthogonality $\perp$ determined by $\xi$
is defined by the condition
\begin{ctext}
  $U_1\perp U_1 \iff \xi(U_1,U_2)=0$, for $U_1,U_2 \in \Sub({\field V})$.
\end{ctext}
For $u,v \in V$ and $U\in\Sub({\field V})$
we write $u\perp U$ if $\gen{u}\perp U$ and $u \perp v$ if $\gen{u}\perp\gen{v}$
i.e. if $\xi(u,v)=0$.
Set $U^\perp := \left\{ u\in V\colon U\perp u \right\}$
and then $\Rad(U) := U \cap U^\perp$.
Write $\rdim(U) = \dim(\Rad(U))$.

The structure $(\fixproj,\perp)$ is called 
{\em a symplectic (metric-)projective space}.

Let $U\in\Sub(\field V)$.
The subspace $U$  is {\em isotropic} 
(the terms {\em totally isotropic} and {\em singular} are also used) if $U \perp U$;
if $U$ is isotropic then $\dim(U)\leq m$.
The subspace $U$ is {\em regular} iff $\Rad(U)$ is the zero-subspace 
of $\field V$.
In other words, $U$ is regular if the form $\xi\restriction{U}$ is nondegenerate.
Let $\Quadr$ stand for the class of isotropic subspaces of $\field V$
and 
let $\regsuby$ stand for the class of regular subspaces of $\field V$.
Since $\xi$ is symplectic, $\regsuby_k\neq\emptyset$ yields $2\dzieli k$.
In particular, $\regsuby_1 = \emptyset$ and $\Quadr_1 = \Sub_1({\field V})$.
Moreover, $\Sub_2(\field V) = \Quadr_2 \cup \regsuby_2$;
i.e. a line of \fixproj \ is either regular or isotropic.

A subspace $U$ of $\field V$ will be called 
{\em tangential}  (cf. \cite{tangi})
if $\Rad(U)$ is a projective point i.e. if $\rdim(U) = 1$.
Let $\tangsuby$ be the class of the tangential subspaces of $\field V$.
It is evident that $\tangsuby_k\neq\emptyset$ yields 
$2 \ndzieli k$.

The following evident observation is worth to note
\begin{fact}\label{fct:postac:tangi}
  Let $U\in\Sub({\field V})$. The following conditions are equivalent:
  \begin{enumerate}[\rm(i)]\itemsep-2pt
  \item
    $U \in\tangsuby$;
  \item
    there are $U_0\in\regsuby$ and a point $p\subset U_0^\perp$ such that
    $U = U_0 + p$.
  \end{enumerate}
\end{fact}

\begin{fact}[{\cite[Lemma 1.2, Cor. 1.3]{grasregul}}]
\label{fct:reghipcia}
  Let $Y \in \Sub({\field V})$. 
  If\/ $Y \supset X\in\regsuby$  
  ($Y \subset X\in\regsuby$) then 
    $\rdim(Y) \leq \codim_{Y}(X)$ ($\rdim(Y) \leq \codim_{X}(Y)$, resp.).
  \par
  Consequently, 
  if $Y$ contains a regular hyperplane 
  or $Y$ is a hyperplane of a regular subspace then 
  $Y$ is tangential.   
\end{fact}

Let us consider the incidence geometry 
  $\Upsilon = \left( \symsuby_k\colon k = 0,\dots, n \right)$,
where 
\begin{ctext}
  $\symsuby_k = \left\{ \begin{array}{ll}
    \regsuby_k & \text{when } 2\dzieli k \\ \tangsuby_k & \text{when } 2 \ndzieli k
\end{array}
\right.$.
\end{ctext}
In a more concise way we can simply write
\begin{equation}\label{def:symsuby}
  \symsuby = \left\{ U\in\Sub({\field V})\colon \rdim(U)\leq 1 \right\}.
\end{equation}
Our incidence geometry $\Upsilon$ is a quasi Curtis-Phan-Tits as introduced 
in \cite[Sec.~3]{BH08} (cf. also \cite{hipsympl}, \cite{Hal89}, \cite{Hal88}).
The family $\symsuby_k$ is the set of all objects of type $k$ in this geometry.
\\
Note that the family $\symsuby$ remains invariant under the map
\begin{ctext}
  $\varkappa\colon\Sub({\field V})\ni U \longmapsto U^\perp$.
\end{ctext}
We adopt the following convention.
\begin{itemize}\def\labelitemi{--}
\item
  Let $H \in\Sub_{k-1}(\field V)$, $B\in\Sub_{k+1}({\field V})$.
  \par
  A \emph{projective top}\/ is a set of the form 
  $\atopof(B) = \left\{ U\in\Sub_k({\field V}) \colon U\subset B \right\}$
  and 
  a \emph{projective star} is a set of the form 
  $\astarof(H) = \left\{ U\in\Sub_k({\field V}) \colon H\subset U \right\}$.
  For $H\subset B$  a \emph{projective pencil} is a set of the form  
  $\apek(H,B) = \atopof(B)\cap\astarof(H)$.
  The class of projective pencils will be denoted by
    $\apeki_k = \apeki_k({\field V})$.
\item
  Let $H\in\symsuby_{k-1}$, $B\in\symsuby_{k+1}$.
  \par
  A \emph{$\symsuby$-top}\/ is a set of the form
  $\topof(B) = \left\{ U\in\symsuby_k \colon U\subset B \right\}$
  and 
  a \emph{$\symsuby$-star} is a set of the form 
  $\starof(H) = \left\{ U\in\symsuby_k \colon H\subset U \right\}$.
  Clearly, 
    $\topof(B)\subset\atopof(B)$ and $\starof(H)\subset\astarof(H)$.
  For $H\subset B$ a \emph{$\symsuby$-pencil} is a set of the form 
  $\pek(H,B) = \topof(B)\cap \starof(H)$.
  The class of {\em nonempty} $\symsuby$-pencils will be denoted by 
    $\peki_k = \peki_k(\symsuby)$.
\end{itemize}
From \ref{fct:reghipcia} we immediately get
\begin{fact}\label{fct:senspek}\strut
\begin{sentences}\itemsep-2pt
\item
  Let $H\in\regsuby_{k-1}$, 
  $H \subset B\in\regsuby_{k+1}$. Then 
  $\apek(H,B) \subset \tangsuby$. 
\item
  Let 
    $H\in\Sub_{k-1}({\field V})$,  $H\subset B\in\Sub_{k+1}({\field V})$,
  and 
    $\apek(H,B)\cap\regsuby_k \neq \emptyset$.
  Then 
    $H,B\in\tangsuby$.
\end{sentences}
\end{fact}
In view of \ref{fct:senspek}, following a standard way 
(cf. \cite{grasregul}, \cite{cohen})
one can construct  
the {\em Grassmann space} (or the {\em space of pencils} in other words) 
\begin{ctext}
  $\PencSpace(\symsuby,k) = \struct{\symsuby_k,\peki_k(\symsuby)}$ 
\end{ctext}
for every integer $1 \leq k \leq n-1$.
In particular (cf. \ref{fct:postac:tangi}), each regular subspace and each
tangential subspace is in a pencil, and
each pencil, if nonempty, has at least two elements.
Therefore, $\PencSpace(\symsuby,k)$ is a partial linear space.
The map $\varkappa$ yields an isomorphism of 
$\PencSpace(\symsuby,k)$ and $\PencSpace(\symsuby,{n-k})$.
Let us also point out that our Grassmann space $\PencSpace(\symsuby,k)$ can be
viewed as a specific shadow space of the Curtis-Phan-Tits geometry $\Upsilon$
(cf. \cite[Sec.~4]{cohen}).

%
%
In view of \ref{fct:senspek}, there is no way to define a space of pencils
associated with regular subspaces alone, i.e. with the incidence geometry 
  $\left( \regsuby_k\colon 1 \leq k\leq n-1 \right)$.
That is why we need the geometry with both regular {\em and} tangential subspaces
in symplectic case.

  Clearly, the notion of a tangential subspace makes sense in arbitrary metric 
  projective geometry determined by a reflexive nondegenerate form (cf. \cite{tangi}).
  In the case of a symmetric form $\xi$ one can consider incidence geometries 
  with tangential subspaces of arbitrary dimension only, or 
  with tangential {\em and} isotropic,
  or with tangential {\em and} regular; 
  in each case a reasonable incidence geometry arises and
  sensible Grassmann spaces can be investigated. Many of the lemmas which we prove in the 
  paper remain valid for a symmetric form $\xi$. However, we do not intend to
  develop the general theory of Grassmannians of regular and tangential subspaces.
  In essence, our intention is to show how to enrich the class of regular subspaces
  of  a {\em symplectic} geometry, as easily as possible, to be able to
  construct a reasonable Grassmann space over it and obtain analogous results as
  in the case of symmetric form and regular subspaces alone.

\smallskip
For $U, W\in\symsuby_k$ we write $U\badjac W$ when 
$U, W$ are collinear in $\PencSpace(\symsuby,k)$.
The goal of this paper is to prove
\begin{thm}\label{thm:main}
  Let $k \neq n - k$. Then the underlying projective symplectic geometry
  $(\fixproj,\perp)$ can be defined in terms of the adjacency $\badjac$ on $\symsuby_k$;
  consequently, $(\fixproj,\perp)$ can be defined in terms of
  the geometry of 
  its Grassmann space $\PencSpace(\symsuby,k)$. 
  In particular, when $2 \dzieli k$, then $(\fixproj,\perp)$ can be defined in terms
  of the adjacency $\badjac$ on its regular $k$-subspaces $\symsuby_k = \regsuby_k$.
\end{thm}
For this purpose we first prove
\begin{prop}\label{prop:main}
  Let $k \neq n-k$.
  If\/ $k > 1$ then  $\PencSpace(\symsuby,k-1)$ 
  can be defined in terms of the binary adjacency relation $\badjac$
  on $\symsuby_k$ and consequently, it can be defined in terms of\/ $\PencSpace(\symsuby,k)$.
  If\/ $k < n-1$ then $\PencSpace(\symsuby,k+1)$
  can be defined in terms of the binary adjacency relation $\badjac$
  on $\symsuby_k$ so, it can be defined in terms of\/ $\PencSpace(\symsuby,k)$.  
\end{prop}

In case $k =1$, the Grassmann space $\PencSpace(\symsuby,1)$ is isomorphic to 
the structure 
$\fixprojr = \struct{\symsuby_1,\symsuby_2,\subset}$,
which is a {\em copolar space} embedded in \fixproj \ 
(cf. \cite{copolar}, \cite{hipsympl}, \cite{embcopol}, and \cite{duaf:copol});
it is simply the line-complement of the corresponding 
polar space 
  $\struct{\Sub_1({\field V}),\Quadr_2,\subset}$.
Informally, we can say that the incidence geometry \eqref{def:symsuby}
is defined over subspaces of \fixprojr \ and write
  $\PencSpace(\symsuby,k) = \PencSpace(\fixprojr,k-1)$.
This approach, however, makes it hard to characterize tangential and regular 
subspaces 
in the language of pure incidence structure \fixprojr.

The validity of Theorem \ref{thm:main} in case $k =1$ is a consequence
of elementary properties of polar spaces (cf. \eqref{krok:start} in Sec.~\ref{sec:proofs}).
Therefore, by induction, 
Theorem \ref{thm:main} follows from \ref{prop:main}.
As a corollary to \ref{thm:main} we obtain
\begin{cor}[A variant of Chow Theorem]\label{cor:main}
  Let $k \neq n - k$. 
  The three classes of maps of\/ $\symsuby_k$:
  \begin{itemize}\def\labelitemi{-}\itemsep-2pt
  \item
    the automorphisms of\/ $\PencSpace(\symsuby,k)$,
  \item
    the bijections which preserve the adjacency $\badjac$ (in both directions),
  \item
    the collineations of\/ $\fixproj$ preserving the conjugacy $\perp$ and
    acting on $\symsuby_k$.
  \end{itemize}
  all coincide.
  \par
  If $k = n - k$ then a bijection of $\symsuby_k$ which preserves $\badjac$ is either 
  determined  by an automorphism of $(\fixproj,\perp)$ or it is a composition of the duality 
  $\varkappa$ and the map determined by an automorphism of $(\fixproj,\perp)$.
\end{cor}

Frequently (cf. e.g. \cite{polargras}, \cite{afpolar}), 
dealing with incidence geometry like \eqref{def:symsuby}, two more adjacencies
are considered: for $U, W\in\symsuby_k$ we write
  $U\botadjac W$ when $U\cap W\in\symsuby_{k-1}$,
and 
  $U\topadjac W$ when $U\neq W$ and 
  $U, W\subset B$ for some $B\in\symsuby_{k+1}$.
Clearly, 
  $U\badjac W$ iff $U\botadjac W$ and $U\topadjac W$.
If $2 \dzieli k$ then 
  $\badjac$, $\botadjac$, and  $\topadjac$ coincide on $\symsuby_k$.
(cf. \eqref{krok:5} in Sec.~\ref{sec:proofs}).
However, if $2 \ndzieli k$ then the relations $\badjac, \botadjac, \topadjac$
on $\symsuby_k$ are pairwise distinct.
To complete the results, we prove also 
the following.
%
%
\begin{thm}\label{thm:updown}
  Let $2 \ndzieli k$. 
  and $1 \leq k \leq n-1$. 
  For $k\neq 1$, the underlying projective symplectic geometry can
  be defined in terms of $\botadjac$ on $\tangsuby_k$, 
  and for $k\neq n-1$ in terms of $\topadjac$ 
  on $\tangsuby_k$.
  Consequently, a bijection of $\tangsuby_k$ which preserves in both
  directions the relation $\botadjac$ or 
  preserves in both directions the relation $\topadjac$
  is determined by an automorphism of the underlying projective symplectic space.
\end{thm}
After a quite technical analysis of geometry of Grassmannians 
$\PencSpace(\symsuby,k)$ in the first two subsections of Sec.~\ref{sec:proofs},
we give complete proofs of Proposition \ref{prop:main}, 
Theorems \ref{thm:main}, \ref{thm:updown}, and Corollary \ref{cor:main} 
in Subs.~\ref{subsec:proofs}. 
Some results, e.g.  \eqref{krok:1}, \eqref{krok:2}, \eqref{krok:6},
\eqref{krok:con4}, and \eqref{krok:30}, seem to be interesting also on their own
right.

%
\section{Technical details and proof of Theorems}\label{sec:proofs}

%
\subsection{Triangles and planes of a symplectic copolar space}

In the remainder of this paper we shall need some more information on triangles 
of symplectic copolar spaces. We gather these in the following list. Most of
them are folklore, but some are given with a short proof for completeness sake.
\begin{fact}\label{fct:copol}\strut
\begin{sentences}
\item\label{copol:1}
  Let $a_1,a_2$ be distinct projective points. 
  Then 
    $a_1,a_2$ are collinear in \fixprojr \  iff $a_1 \not\perp a_2$.
\item\label{copol:2}
  A projective plane $\pi$ contains a triangle of \fixprojr \ 
  iff $\rdim(\pi)= 1$ i.e. iff $\pi\in\tangsuby$. 
\item\label{copol:3}
  If\/ $\rdim(\pi)  =1$ then each projective line on $\pi$ is either isotropic and passes
  through $p = \Rad(\pi)$, or it is regular and misses $p$.
\item[\strut]
  \begin{normalfont}
  Let $\pi\in\tangsuby_3$ (i.e. let $\pi$ be a tangential plane).
  We write
    $\pi^\infty = \Rad(\pi)$ and $[\pi] = \pi\setminus\{ \pi^\infty \}$.
  \end{normalfont}
\item\label{copol:4}
  Let a plane $\pi\in\tangsuby$ contain a triangle $\Delta$ with regular sides. 
  Then
  the set $\overline{\Delta}$ of points on lines of \fixprojr \ 
  which cross all the sides of $\Delta$ coincides with $[\pi]$.
  {\normalfont This can be read as follows: 
  a triangle spans a plane of \fixprojr, which is a dual
  affine plane, cf. \cite{hipsympl}, \cite{gramlich}, and \cite{duaf:copol}.}
\item\label{copol:5}
  Let planes $\pi_1,\pi_2\in\tangsuby$ have a regular line $L$ in common.
  There is a tetrahedron  
  with regular edges, whose one edge is $L$, one face is
  in $\pi_1$, and other in~$\pi_2$.
  \begin{proof}
    Let $p_i = \Rad(\pi_i)$ for $i=1,2$. Then $p_1,p_2\notin L$, since otherwise
    (say, $p_1\in L$), 
    we would get $p_1 \in L \perp p_1$, so $p_1\in\Rad(L)$.
    Take any 
      $a_1\in\pi_1\setminus (L\cup \{ p_1 \})$.
    Suppose that 
      $a_1 \perp (\pi_2 \setminus (L\cup \{ p_2 \}))$.
    Then 
    $a_1 \perp \pi_2$ and thus $a_1\perp L$, 
    which gives $a_1\perp L + a_1 = \pi_1$  and $a_1 = p_1$. 
    Consequently, there is 
      $a_2\in \pi_2 \setminus (L\cup \{ p_2 \})$
    with $a_1 \not\perp a_2$ 
    and then $a_1,a_2$ are on a regular line.
    One can easily find $a_3,a_4 \in L$ 
    such that $a_1,a_2 \not\perp a_3,a_4$.
  \end{proof}
\item[\strut]
  \begin{normalfont}
  Write $\pi_1 \blisk \pi_2$ iff there is a tetrahedron as in \eqref{copol:5}
  \end{normalfont}
\item\label{copol:6}
  Let planes $\pi_1,\pi_2\in\tangsuby$ have an isotropic line $L$ in common.
  Let 
    $a_3,a_4 \in L$, $a_3,a_4 \neq \Rad(\pi_1),\Rad(\pi_2)$, and $a_3 \neq a_4$.
  Then there are 
    $a_i\in\pi_i\setminus L$ such that $a_1\not\perp a_2$
    and $a_1,a_2 \not\perp a_3,a_4$, so 
  $a_1,a_2,a_3$ and $a_1,a_2,a_4$ span two planes $\pi'_1,\pi'_2\in\tangsuby$ with common 
  regular line $M = a_1 + a_2$.
  {\normalfont By the above, $\pi'_1 \blisk \pi'_2$.}
  \begin{proof}
    From assumption, $L\perp L$.
    Write $p_i = \Rad(\pi_i)$. Then $p_1,p_2 \in L$. 
    Let $a_1 \in \pi_1\setminus L$  be arbitrary. 
    Then $p_1 \perp a_1$; 
    if $a_1 \perp a_3$ or $a_1 \perp a_4$ then  $a_1 \perp L$ and thus $a_1 \perp \pi_1$. 
    As above, we find 
    $a_2 \in \pi_2\setminus L$ with $a_1\not\perp a_2$ and we are through.
  \end{proof}
\item[\strut]
  \begin{normalfont}
  Write $\pi_1 \bliskx \pi_2$ iff there are planes $\pi'_1, \pi'_2$ as in \eqref{copol:6}.
  \par
  In view of \eqref{copol:5} and \eqref{copol:6}, we obtain
  \end{normalfont}
\item\label{copol:7}
  Let $\pi_1,\pi_2$ be two tangential planes. If 
    $|[\pi_1]\cap[\pi_2]|\geq 2$ then 
    $\pi_1 \blisk \pi_2$ or $\pi_1 \bliskx \pi_2$.
\item\label{copol:8}
  Any two planes in $\tangsuby$ can be joined by a sequence $(\pi_i\colon i=0,\dots, t)$ 
  of tangential planes
  such that 
    $|[\pi_{i-1}]\cap[\pi_i]|\geq 2$ for $i=1,\dots, t$.
  %
  \begin{proof}
    By \eqref{copol:1}, any two points $a',a''$ can be joined in \fixprojr \ by a
    path (sometimes also called a polygonal path, which by the way is of length $\leq 2$ here).
    Let $a_0,\dots, a_t$ be such a path in $\fixprojr$ that joins $a',a''$
    and let $L_i$ be the line through $a_{i-1},a_{i}$ for $i=1,\dots, t$.
    By \eqref{copol:4}, any two consecutive sides $L_{i},L_{i+1}$ lie in a 
    plane $\pi_i\in\tangsuby$, $i=1,\dots, t-1$.
    Let $\pi',\pi''\in\tangsuby_2$; take regular lines 
    $L'$ in $\pi'$ and $L''$ in $\pi''$ and points $a'$ on $L'$, $a''$ on $L''$.
    Taking 
      $\pi_0 = L'+L_1$ if $L'$ is not on $\pi_1$ and $\pi_0 = \pi_1$ otherwise,
    and similarly, 
      $\pi_t = L_t+L''$ or $\pi_t = \pi_{t-1}$
    we obtain a desired sequence of planes.
  \end{proof}
\item\label{copol:9}
  For each point $p$ of $\fixprojr$ there is a tangential plane $\pi$ such that 
  $p \in [\pi]$.
  \begin{proof}
    First, we take any point $q$ such that $q\not\perp p$ and let $L$ be a line
    through $p,q$. Next, let $\pi$ be a plane that contains $L$. Then $L$ is regular,
    so $\pi$ is tangential. By \eqref{copol:3}, $L$ misses $\Rad(\pi)$
    and thus $p \neq \Rad(\pi)$.
  \end{proof}
\item[\strut]
  \begin{normalfont}
  Let us write $\blisq$ for the transitive closure of the relation
  $\blisk \cup \bliskx$.
  By \eqref{copol:7}, \eqref{copol:8}, and \eqref{copol:9} we obtain
  \end{normalfont}
\item\label{copol:10}
  Let $\Delta$ be any triangle of $\fixprojr$. Then 
  the set of points of\/ $\fixprojr$ is the union
  \begin{equation}\label{eq:rozpiecie:1}
    \widetilde{\Delta} = \bigcup\left\{ \overline{\Delta'} \colon \;\;
    \overline{\Delta'}\blisq \overline{\Delta},\;
    \Delta' \text{ is a triangle} \right\}.
  \end{equation}
\item[\strut]
  \begin{normalfont}
  Let $a_1,a_2$ be points of \fixprojr.
  Write $a_1 \adjac a_2$ when $a_1,a_2$ are collinear in \fixprojr.
  In view of \eqref{copol:1}, 
  {\em the orthogonality of points is definable in \fixprojr:
  $a_1 \perp a_2$ iff $a_1 = a_2$ or $a_1\nadjac a_2$}.
  \end{normalfont}

\item\label{copol:12}
  Let\/ $a_1,a_2,b_1,b_2$ be four points of\/ $\fixprojr$ such that
    $a_1\adjac a_2$, $b_1\nadjac b_2$, and $a_1,a_2\adjac b_1,b_2$
  (cf. \eqref{copol:6}).
  Then there are points $c_1,c_2$ such that 
    $c_1\nadjac c_2$  and $a_1,a_2,b_1,b_2 \adjac c_1,c_2$.
  \begin{proof}
    Let $L$ be the line through $a_1,a_2$.
    Suppose, first, that $b_1,b_2$ are on a plane $\pi$ through $L$.
    Then $\pi$ is tangential; let $p = \Rad(\pi)$, so the line $M$ through $b_1,b_2$
    is isotropic and it passes through $p$. Clearly, $b_1,b_2\neq p$, as the lines 
    $\LineOn(b_1,a_i)$ and $\LineOn(b_2,a_i)$ are not isotropic.
    We take any line $M_0 \neq M$ through $p$ that misses $a_1,a_2$.
    Any two  points $c_1, c_2$ on $M_0$ that are distinct from $p$ and not lie on $L$
    satisfy our requirements.
    \par
    If there is no plane through $L$ which contains $b_1,b_2$ then
    the lines $L$ and $M$ are skew and they span a projective $3$-space $\Gamma$
    (linearly, $\dim(\Gamma) = 4$); 
    as $\Gamma$ contains $L$, $\rdim(\Gamma)\leq 2$, and 
    as $\Gamma\diagup\Rad(\Gamma)$ is nondegenerate $2 \dzieli (4 - \rdim(\Gamma))$. 
    Consequently, there are two cases to consider.
    Firstly, if $\Rad(\Gamma)$ is a line $K$ then isotropic lines in $\Gamma$ are exactly the 
    projective lines which cross $K$; in particular, $M$ crosses $K$ in a point $p$.
    Let $\pi_i = K + a_i$ for $i =1,2$; 
    then $M$ is neither contained in $\pi_1$ nor in $\pi_2$.
    Let $\pi$ be a plane that contains $M$ and does not contain $K$; then 
    $\pi\cap K = p$. Let $M_0$ be other line in $\pi$ that passes through $p$
    and is not contained in $\pi_1,\pi_2$.
    For each $i=1,2$ there is at most one point on $M_0$ which is not collinear
    with $a_i$ (as otherwise, $M_0$ is contained in $\pi_i$).
    Thus one can find on $M_0$ two points $c_1,c_2$ distinct from $p$
    and collinear with $a_1,a_2$. 
    From construction, 
      $b_1,b_2 \adjac c_1,c_2$ and $b_1\not\adjac c_2$.
%
    Secondly, let $\Gamma$ be regular; then 
    $\ind(\xi\restriction{\Gamma})\leq 2$, so
    each plane in $\Gamma$ is tangential. 
    Let 
    $\pi_1$ be the plane spanned by $a_1,a_2,b_2$, 
    $M$ be a line through $p = \Rad(\pi_1)$ contained in $\pi_1$ and  missing $a_1,a_2,b_2$,
    $\pi_2$ be the plane through $M,b_1$, and 
    $q = \Rad(\pi_2)$. 
    Since $M\subset \pi_2$ is isotropic, by \eqref{copol:3}, $q \subset M$. 
    Take $c_1,c_2$ on $M$ distinct from $p,q$; then $c_1\not\adjac c_2$.
    By \eqref{copol:3}, $c_1,c_2 \adjac a_1,a_2,b_1,b_2$.
  \end{proof}

\item\label{copol:13}
  Let $a_1,a_2,a_3$ be a triangle in\/ $\fixprojr$.
  There are points $b_1,b_2,b_3$ such that 
    $a_1,a_2,a_3\adjac b_1,b_2,b_3$,\; 
    $b_1\not\adjac b_2,b_3$, and\/
    $b_2\adjac b_3$.
  \begin{proof}
    Let $\pi_1$ be the plane spanned by $a_1,a_2,a_3$. Then $\pi_1\in\tangsuby_3$
    and one can find $\Gamma\in\regsuby_4$ with $\pi\subset \Gamma$. 
    Let us restrict to  $\Gamma$ considered as a $3$-dimensional symplectic 
    projective space.
    As above, each plane contained in $\Gamma$ is tangential. 
    One can complete in $\Gamma$ given triangle to a tetrahedron
    $a_1,a_2,a_3,b_3$. 
    Let us take on $\pi_1$ a line $L$ through $p = \Rad(\pi_1)$
    which misses $a_1,a_2,a_3$; 
    let $\pi_2$ be the plane through $L$, $b_3$,
    and let $q = \Rad(\pi_2)$. 
    $\pi_2$ contains an isotropic line $L$, so $L$ contains $q$. 
    Suppose that $p = q$; 
    then $p\perp\pi_1,\pi_2$ gives $p\perp \Gamma$, a contradiction. 
    Set $b_1 := q$. Then $b_1\adjac a_1,a_2,a_3$.
    Let $b_2$ be a point on $L$ distinct from $p,q$. 
    By \eqref{copol:3}, 
    considering  $\pi_1$ we get 
      $b_2 \adjac a_1,a_2,a_3$, $b_1\not\adjac b_2$,
    and considering $\pi_2$ we get 
      $b_2\adjac b_3$, $b_3\not\adjac b_1$.
  \end{proof}
\end{sentences}
\end{fact}

Now we are going to reconstruct the Grassmann space $\PencSpace(\symsuby,k)$ 
in terms of binary adjacency $\badjac$. This involves the family of maximal
cliques of $\badjac$, i.e. stars and tops, and we shall show that they are
definable and distinguishable in the first place. 
Our reasoning depends on whether $k$ is even or odd.

\subsection{Case $2\ndzieli k$}

Let 
  $H \in\regsuby_{k-1}$, $H \subset B\in\regsuby_{k+1}$.

Let 
  $U\in\astarof(H)$ i.e. let $H\subset U\in\Sub_k({\field V})$.
From \ref{fct:reghipcia}, $\rdim(U) = 1$. 
Consequently,
\begin{ctext}
  $\starof(H) = \astarof(H)$.
\end{ctext}
Let $q = \Rad(U)$, then $q \perp U$;
in particular, $q \perp H$.
Since $H\in\regsuby$, $q \notin H$ and thus $U = H \oplus q$ with $q \in H^\perp$.
Consequently, an element of $\starof(H)$  
is any subspace $U$ of the form as above i.e.
\begin{ctext}
  $\starof(H) = \left\{ H+q\colon q \text{ is a point}, \;
  q \subset H^\perp \right\}$.
\end{ctext}
Thus the elements of $\starof(H)$ can be identified with the elements
of $\Sub_1(H^\perp)$ i.e. with the points of the nondegenerate metric projective
space defined over $H^\perp$.

Analogous remarks concern the geometry of 
$\topof(B) = \atopof(B)$. 
In particular, we can write
\begin{ctext}
  $\topof(B) = \left\{ B\cap q^\perp\colon q \text{ is a point on } B \right\}$.
\end{ctext}
From the above we get:
\begin{efekt}\label{krok:peki1}
  if $H\subset B$ then $\starof(H)\cap\topof(B) = \apek(H,B)$, so
  it contains at least two elements.
\end{efekt}

One can note that 
\begin{ctext}
  $\pek(H,B)
  = \left\{ H\oplus\gen{u}\colon u \in P,\; u\neq\theta \right\}$,
  where $P = H^\perp \cap B$, and $P\in\regsuby_2$.
\end{ctext}

The form $\xi$ restricted to $H^\perp$ is symplectic, so
the points of the metric projective space $\starof(H)$
are simply the points of the 
corresponding projective space $\PencSpace(H^\perp,1)$; the distinction lies in
the adopted family of lines.
These lines correspond to suitable pencils i.e. to $B\in\regsuby_{k+1}$ with $H\subset B$.
We see that the $B$ above correspond to regular $2$-subspaces of $H^\perp$
and thus 
\begin{efekt}\label{krok:1}
  the geometry of the restriction of\/ 
  $\PencSpace(\symsuby,k)$ to\/ $\starof(H)$ is 
  an $(n-k)$-dimensional symplectic copolar space.
\end{efekt}
With analogous reasoning we obtain that 
\begin{efekt}\label{krok:2}
  the restriction of\/ 
  $\PencSpace(\symsuby,k)$ to\/ $\topof(B)$ is a $k$-dimensional symplectic 
  copolar space.
\end{efekt}

Let $\Delta$ be a triangle in $\PencSpace(\symsuby,k)$.
Then, by common projective geometry, its vertices are either in a top $\topof(B) =:{\cal X}$
or in a star $\starof(H) =: {\cal X}$. Clearly, $\overline{\Delta}$, as defined in 
\ref{fct:copol}\eqref{copol:4} with $\fixprojr$ replaced by $\PencSpace(\symsuby,k)$, 
is a plane in $\cal X$.
Note that if $\Delta'$ is a different triangle in $\PencSpace(\symsuby,k)$
and 
  $\overline{\Delta'} \blisk \overline{\Delta}$ or 
  $\overline{\Delta'} \bliskx \overline{\Delta}$
(cf. \ref{fct:copol}), 
then 
$\Delta'$ lies in $\cal X$ as well.
From the above, \eqref{krok:1}, \eqref{krok:2},
and \ref{fct:copol}\eqref{copol:10}, we get that
\begin{ctext}
  $\widetilde{\Delta} = {\cal X}$,
\end{ctext}
$\widetilde{\Delta}$ being the set defined in $\PencSpace(\symsuby,k)$ by formula \eqref{eq:rozpiecie:1}.
Finally, we get that 
\begin{efekt}\label{krok:3}
  the family of stars and tops 
  is definable in terms of geometry of $\PencSpace(\symsuby,k)$;
  by \eqref{krok:1} and \eqref{krok:2},
  if $k \neq n-k$, stars and tops are intrinsically distinguishable.
\end{efekt}

\medskip
Let $U_1,U_2$ be distinct points of $\PencSpace(\symsuby,k)$ such that
  $U_1 \badjac U_2$,
and let 
  $U_1,U_2 \in g = \pek(H,B)\in\peki_k(\symsuby)$.
Next, let $W_1, W_2$ satisfy 
\begin{ctext}
  $W_1, W_2\badjac U_1,U_2$\quad and\quad $W_1\not\badjac W_2$.
\end{ctext}
Consider the set
\begin{ctext}
  $\agen{U_1, U_2, W_1, W_2}{\badjac} := \bigl\{ U\colon U\badjac U_1, U_2, W_1, W_2 \bigr\}$.
\end{ctext}
From common projective geometry, as $\badjac$-neighbor subspaces are adjacent
in the projective Grassmannian over \fixproj,
there are three possibilities to consider
\begin{itemize}\def\labelitemi{--}\itemsep0pt
\item
  $W_1, W_2\in\starof(H)$. 
  By \ref{fct:copol}\eqref{copol:12} and \eqref{krok:1}, there are 
    $U, W\in\agen{U_1, U_2, W_1, W_2}{\badjac}$ 
  such that 
    $U\not\badjac W$.

\item
  $W_1, W_2\in\topof(B)$. 
  Analogously, there are 
    $U, W\in\agen{U_1, U_2, W_1, W_2}{\badjac}$
  such that 
    $U\not\badjac W$.

\item
  $W_1\in\topof(B)\setminus\starof(H)$ and $W_2\in\starof(H)\setminus\topof(B)$
  or 
  $W_2\in\topof(B)\setminus\starof(H)$ and $W_1\in\starof(H)\setminus\topof(B)$.
  In that case 
    $\agen{U_1, U_2, W_1, W_2}{\badjac}\subset \pek(H,B)$
  and thus 
    $U\badjac W$ for all $U, W\in\agen{U_1, U_2, W_1, W_2}{\badjac}$.
\end{itemize}
Let $U_3\in g$ be arbitrary; clearly, one can find 
  $W_1\in\topof(B)\setminus\starof(H)$,
  $W_2\in\starof(H)\setminus\topof(B)$ 
such that 
  $U_1,U_2,U_3\badjac W_1, W_2$.
Write $\wspolin$ for the ternary collinearity relation of 
$\PencSpace(\symsuby,k)$. In view of the above analysis, the formula
\begin{multline}\label{def:adjac2col}
  \wspolin(U_1, U_2, U_3) \iff U_1 \badjac U_2 \Land (\exists{W_1, W_2})\,
  \Bigl[\;
  W_1, W_2\badjac U_1,U_2 \Land W_1\not\badjac W_2 \Land 
  \\
  (\forall{U, W})\,
  \bigl[\, U_1, U_2, W_1, W_2\badjac U, W \implies U\badjac W \,\bigr]
  \Land U_3 \badjac U_1, U_2, W_1, W_2
  \;\Bigr]
\end{multline}
defines the relation $\wspolin$ in terms of the adjacency $\badjac$
for distinct $U_1,U_2$.
Finally, we conclude that 
\begin{efekt}\label{krok:4}
  $\PencSpace(\symsuby,k)$ is definable in terms of the adjacency $\badjac$.
\end{efekt}

\subsection{Case $2 \dzieli k$}

Let $H \in\tangsuby_{k-1}$, $B\in\tangsuby_{k+1}$.
Set $p = \Rad(H)$ and $q = \Rad(B)$.

Note that $U\in\Sub_k(B)$ is regular iff $U$ is a linear complement of $q$ in $B$
i.e. iff $B = q \oplus U$.
This can be written as
\begin{equation}\label{eq:topy2}
  \topof(B) = \left\{ U\in\regsuby_k\colon U\subset B \right\} =
  \left\{ U\in\Sub_k(B)\colon q \not\subset U \right\}.
\end{equation}
It is easy to note that
\begin{equation}\label{eq:gwiazdy2}
  \starof(H) = \left\{ U\in\regsuby_k\colon H\subset U \right\}
  = \left\{ H+\gen{u}\colon u \notin p^\perp \right\}.
\end{equation}
Indeed, let $U = H\oplus\gen{u}$ with $u\neq\theta$.
Suppose that $p \perp u$; then $p \perp U$ and thus $U\notin\regsuby$.
Conversely, let $w\in\Rad(U)$, $w\neq\theta$. Then $w \in U$, $w \perp U$, so $w \perp p$.
If $w\notin H$ then $p\perp H+\gen{w} = U$ and thus $p\perp u$.
If $w \in H$ then $w\perp U\supset H$ gives $w = p$; 
we get $p\perp U$ and thus $p\perp u$.

Assume that  $H \subset B$. 
There are two cases to consider
\begin{sentences}\leftmargin0pt\itemindent13ex
\item[$q \in H$:] 
  Then $q \perp B\supset H$ gives $p=q$.
  For arbitrary $U\in\apek(H,B)$ we have then $p\in \Rad(U)$ and thus
  %
    $\pek(H,B) = \emptyset$.
\item[$q \notin H$:]
  Then $p\neq q$. Set 
    $L := p + q$ and $U_0 := H + q = H + L$. 
  Note that $L\perp L$,
  so $L\in\Quadr_2$; moreover, $L\perp H$. 
  From this we obtain $\Rad(U_0) = L$
  (in particular, $\rdim(U_0) = 2$).
  We have
  %
    $\pek(H,B) = \apek(H,B)\setminus\left\{ U_0 \right\}$.
  %
  Indeed, let $U\in\apek(H,B)$ and $U \neq U_0$. Then $q \notin U$ and, 
  by \eqref{eq:topy2},
    $U\in\topof(B)$ and thus $U\in\pek(H,B)$.
\end{sentences}
From the above we have the following:
\begin{efekt}\label{krok:peki2}
  if $H \subset B$ then $\starof(H)\cap\topof(B)$ is either empty or it 
  contains at least two elements.
\end{efekt}

Let $U_1,U_2 \in\regsuby_k$. 
Assume that 
  $U_1\cap U_2\in\Sub_{k-1}({\field V})$.  
Then 
  $U_1 + U_2 \in \Sub_{k+1}({\field V})$; 
in view of \ref{fct:reghipcia}, 
  $U_1\cap U_2, U_1 + U_2 \in \tangsuby$ 
and thus
  $U_1,U_2\in\pek({U_1\cap U_2},{U_1+U_2})$.
Therefore, 
\begin{efekt}\label{krok:5}
  the binary collinearity in $\PencSpace(\symsuby,k)$ 
  coincides with the projective adjacency on $\regsuby_k$.
\end{efekt}

Next, note that 
\begin{efekt}\label{krok:6}
  the geometry of restriction of\/ 
  $\PencSpace(\symsuby,k)$ to a top and to a star is a 
  $k$-di\-men\-sio\-nal and a $(n-k)$-dimensional resp., affine geometry.
\end{efekt}
Indeed,
\begin{sentences}\def\labelitemi{--}\itemindent5ex\leftmargin0pt
\item[\labelitemi]
  The family ${\cal X} := \astarof(H)$ has the natural structure of a projective space.
  The set $\starof(H)$ is obtained by removing the segment 
  $[H,p^\perp]_k = \left\{ U\in\sub_k({\field V}) \colon H \subset U \subset p^\perp \right\}$ 
  from $\cal X$ (cf. \eqref{eq:gwiazdy2}). 
  Computing dimensions we see that $[H,p^\perp]_k$ is a hyperplane in $\cal X$,
  which justifies our claim. The induced affine geometry has dimension $n-k$.
\item[\labelitemi]
  By \eqref{eq:topy2}, 
  $\topof(B)$ consists of the hyperplanes in the projective space $\atopof(B)$ which omit
  a point $q$; this procedure results in a $k$-dimensional affine space.
\end{sentences}

\medskip

Let us note a straightforward consequence of \eqref{krok:6} that
%
%
%
%
\begin{efekt}\label{krok:9}
  stars and tops are distinguishable in terms of $\badjac$ provided
  $k \neq n-k$.
\end{efekt}

\medskip
The sets $\starof(H)$ and $\topof(B)$ both are strong subspaces of 
$\PencSpace(\symsuby,k)$; in particular, they are cliques of the binary collinearity
relation $\badjac$.
Moreover, they are exactly  {\em the maximal} cliques of $\badjac$.
This yields that 
\begin{efekt}\label{krok:7}
  the family of stars and tops is definable in
  terms of the adjacency $\badjac$ on $\symsuby_k$.
\end{efekt}
An intersection of two maximal $\badjac$-cliques has at least two elements
iff one of these cliques is a star and another one is a top, and then their
intersection is a pencil i.e a line of $\PencSpace(\symsuby,k)$.
This justifies that 
\begin{efekt}\label{krok:8}
  the structure $\PencSpace(\symsuby,k)$
  is definable in terms of the adjacency $\badjac$.
\end{efekt}

\subsection{Adjacencies $\botadjac$ and $\topadjac$ on $\tangsuby$}

As we already stated, if $2\dzieli k$, then all the three adjacency relations
$\badjac$, $\botadjac$, and  $\topadjac$ coincide on $\symsuby_k = \regsuby_k$.
In case $k$ is odd we need more elaboration to get the Grassmann space
$\PencSpace(\symsuby,k)$ defined in terms of $\botadjac$, as well as in terms
of $\topadjac$.

Let 
  $2 \ndzieli k$ and $1 < k < n-1$. 
We begin our analysis with 
  $\botadjac$ on $\symsuby_k = \tangsuby_k$.
Note that if $U\botadjac W$ then 
$U, W$ are adjacent in the projective Grassmannian over \fixproj.

Let $U_1,U_2,U_3$ be a $\botadjac$-clique i.e. let them be pairwise distinct
and $U_i \botadjac U_j$ for distinct $i,j =1,2,3$.
Consider the set 
\begin{ctext}
  $\agen{U_1,U_2,U_3}{\botadjac} = \left\{ U \colon U\botadjac U_1,U_2,U_3 \right\}$.
\end{ctext}
By known properties of projective Grassmannians the following possibilities arise.
\begin{sentences}
\item\label{moze1}
  $U_1, U_2, U_3$ are in a projective pencil $g = \apek(H,B)$; then $H \in \regsuby_{k-1}$.
\item\label{moze2}
  \eqref{moze1} fails, but $U_1,U_2,U_3 \in \astarof(H)$ for some $H$; 
  then $H\in\regsuby_{k-1}$. 
  We say that $U_1, U_2, U_3$ yield a $\stars$-triangle.
\item\label{moze3}
  \eqref{moze1} fails, but $U_1,U_2,U_3 \in \atopof(B)$ for some $B$. Since $B$ contains
  a regular subspace $U_1 \cap U_2$ by \ref{fct:reghipcia} we get $\rdim(B)\leq 2$
  and thus two cases are possible:
  \begin{enumerate}[(a)]\itemsep-2pt
  \item\label{moze3a}
    $B \in \regsuby_{k+1}$, 
    we say that $U_1, U_2, U_3$ yield a $\tops$-triangle; or
  \item\label{moze3b}
    $\Rad(B)$ is a line $L$, 
    we say that $U_1, U_2, U_3$ yield a $\topsx$-triangle. 
  \end{enumerate}
\end{sentences}
In case \eqref{moze2} 
{\em the set $\agen{U_1, U_2, U_3}{\botadjac}$ is a $\botadjac$-clique}.
Indeed, if 
  $U, W \in \agen{U_1, U_2, U_3}{\botadjac}$ 
are distinct then 
  $U, W\in\starof(H)$.

In case \eqref{moze3} 
{\em the set $\agen{U_1, U_2, U_3}{\botadjac}$ is not a $\botadjac$-clique}.
Consider \eqref{moze3b} first. 
A $(k-1)$-subspace $H$ of $B$ is regular iff $H$ misses $L$.
On the other hand, a $k$-subspace $U$ is tangential iff it crosses $L$ in a point
(which turns out to be its radical). 
Write $p_i = L\cap U_i$, then 
$p_1,p_2,p_3$ are pairwise distinct, as the intersections $U_i \cap U_j$ are regular.
Take any point $p$ on $L$ distinct from $p_1, p_2, p_3$ 
and two $W_1, W_2$ through $p$ which do not contain $L$; 
then 
  $W_1, W_2\botadjac U_1, U_2, U_3$ 
(the intersections $W_1\cap U_i$, $W_2\cap U_i$ miss $L$) 
but 
  $W_1\not\botadjac W_2$.
Note that if 
  $W_1, W_3 \in \agen{U_1, U_2, U_3}{\botadjac}$ and $W_1\not\botadjac W_3$
then $p \in W_3$ and thus $W_2\not\botadjac W_3$. 
Hence
{\em the relation $\not\botadjac$ is transitive on $\agen{U_1, U_2, U_3}{\botadjac}$}.

Let us pass to case \eqref{moze3a}. 
Then $U_1, U_2, U_3$ can be considered as points
of a suitable symplectic copolar space $\topof(B)= \atopof(B)$ 
and we can use known properties of symplectic projective geometry. 
So, let $\pi$ be the plane spanned in $\topof(B)$ by $U_1, U_2, U_3$. 
To justify that 
$\agen{U_1, U_2, U_3}{\botadjac}$ is not a $\botadjac$-clique it suffices to take 
on $\topof(B)$ any line through the radical of $\pi$ and missing $U_1, U_2, U_3$,
and distinct points $U, W$ on this line.
On the other hand 
by \ref{fct:copol}\eqref{copol:13}, 
%
one can find 
  $W_1, W_2, W_3\in\topof(B)$
such that 
  $U_1, U_2, U_3 \botadjac W_1, W_2, W_3$, \ $W_2\botadjac W_3$,
  and $W_1\not\botadjac W_2, W_3$.
This justifies that 
{\em the relation $\not\botadjac$ is not transitive on $\agen{U_1,U_2,U_3}{\botadjac}$}.

In case \eqref{moze1}, clearly, 
{\em the set $\agen{U_1, U_2, U_3}{\botadjac}$ is not a $\botadjac$-clique}. 
To see this it suffices to find 
  $U\in\starof(H)\setminus g$ and 
  $W\in\atopof(B)\setminus g$ with $U_1, U_2, U_3 \botadjac U, W$.
Moreover, from properties of the copolar space $\starof(H)$,
{\em one can find in $\starof(H)$ a $\botadjac$-clique $W_1, W_2, W_3$ 
which does not fall into case \eqref{moze1} such that 
  $U_1, U_2, U_3\in \agen{W_1, W_2, W_3}{\botadjac}$}.
Moreover, 
{\em if $B\in\regsuby_{k+1}$ then an analogous triple can be found in $\topof(B)$}.

In view of the above analysis, a $\botadjac$-clique $U_1, U_2, U_3$ forms a $\stars$-triangle iff 
$\agen{U_1, U_2 ,U_3}{\botadjac}$ is a $\botadjac$-clique, so
\begin{efekt}\label{krok:21}
  the class of $\stars$-triangles  can be distinguished in terms of $\botadjac$.
\end{efekt}
Now, $\tops$-triangles/$\topsx$-triangles  
can be characterized as $\botadjac$-cliques $U_1, U_2, U_3$ 
which are, firstly, not contained in any $\agen{W_1, W_2, W_3}{\botadjac}$ for 
a $\stars$-triangle $W_1, W_2, W_3$ and, secondly, such that 
${\cal X} = \agen{U_1, U_2, U_3}{\botadjac}$ is not a $\botadjac$-clique 
and $\not\botadjac$ is not transitive/is transitive on $\cal X$.
Consequently,
\begin{efekt}\label{krok:22}
  the class of $\tops$-triangles and the class of $\topsx$-triangles
  can be distinguished in terms of $\botadjac$.
\end{efekt}
Let $\wspolin$ be the ternary collinearity relation of $\PencSpace(\symsuby,k)$.
Similarly as in \eqref{def:adjac2col} we can write 
\begin{multline}\label{def:botadjac2col}
  \wspolin(U_1,U_2,U_3) \iff (\exists{U'_1,U'_2,U'_3,U''_1,U''_2,U''_3})
  \bigl[\;
  U'_1,U'_2,U'_3 \text{ is a }\stars\text{-triangle} \Land
  \\
  U''_1,U''_2,U''_3 \text{ is a }\tops\text{-triangle} \Land
  U_1,U_2,U_3 \in \agen{U'_1,U'_2,U'_3}{\botadjac}\cap
  \agen{U''_1,U''_2,U''_3}{\botadjac}
  \;\bigr].
\end{multline}
By \eqref{krok:21}, \eqref{krok:22}, and \eqref{def:botadjac2col} we obtain that
\begin{efekt}\label{krok:23}
  the structure $\PencSpace(\symsuby,k)$ can be defined in terms of the
  adjacency $\botadjac$.
\end{efekt}

Clearly,  $\varkappa$ maps the relation $\botadjac$ on $\symsuby_k$
onto $\topadjac$ on $\symsuby_{n-k}$. Thus, \eqref{krok:23}  
with $k$ replaced by $n-k$ yields that
\begin{efekt}\label{krok:24}
  the structure $\PencSpace(\symsuby,k)$ can be defined in terms of the
  adjacency $\topadjac$.
\end{efekt}
Finally, \eqref{krok:23} and \eqref{krok:24} enable us to conclude with the following.
\begin{efekt}\label{krok:updown}
  The structure $\PencSpace(\symsuby,k)$ can be defined in terms of both
  $\botadjac$ and $\topadjac$ on $\tangsuby_k$. 
  Consequently, $\badjac$ can be defined in terms of $\botadjac$
  and in terms of $\topadjac$.
\end{efekt}

\subsection{Connectedness}

We use standard methods to show that automorphisms of binary adjacency
preserve, or exchange, the two families of its maximal cliques. These
methods rely on the fact that the adjacency in question is connected.

For a relation $\rho$ on $\symsuby_k$ we say that 
$U, W\in\symsuby_k$ are {\em $\rho$-connected} if there is a sequence
  $U = U_0,\dots, U_t = W$ 
such that 
  $U_{i-1}\mathrel{\rho} U_i$ or $U_{i-1} = U_{i}$ for $i=1,\dots, t$.
\par
By \ref{fct:copol}\eqref{copol:1} we easily get that 
\begin{efekt}\label{krok:con0}
  the binary collinearity of points of a symplectic copolar space is connected.
\end{efekt}
Let 
  $k \leq n-1$,
  $H\in\symsuby_{k-1}$, and $U, W\in\starof(H)$.
If $2 \dzieli k$ then by \eqref{krok:5}, $U\badjac W$.
If $2 \ndzieli k$ then, by \eqref{krok:1} and \eqref{krok:con0},
the points $U, W$ of $\starof(H)$ can be joined in $\starof(H)$ by 
a polygonal path. 
This gives the following.
\begin{efekt}\label{krok:con1}
  Let $U \botadjac W$. Then $U, W$ are $\badjac$-connected.
\end{efekt}
In a consequence, we get the following.
\begin{efekt}\label{krok:con2}
  Let $U, W\in\symsuby_k$  
  and $U\botadjac W$. 
  Then there are $\botadjac$-connected $B', B''\in\symsuby_{k+1}$
  such that $U\subset B'$ and $W\subset B''$.
\end{efekt}
Now, it is easy to prove that
\begin{efekt}\label{krok:con3}
  the conclusions of \eqref{krok:con1} and of \eqref{krok:con2} hold for any
  $\botadjac$-connected $U, W\in\symsuby_k$. 
\end{efekt}
Let $U, W\in\symsuby_k$. Consider
two flags 
  $U_1\subset\dots\subset U_k=U$ and 
  $W_1\subset\dots\subset W_k=W$ 
such that
  $U_i, W_i\in\symsuby_i$.
Note that if 
  $U_i\subset B'$, $W_i\subset B''$ and 
  $B',B''\in\symsuby_{i+1}$ 
then 
  $U_{i+1}\botadjac B'$ and $W_{i+1}\botadjac B''$, 
so 
if $B',B''$  are $\botadjac$-connected then
$U_{i+1}, W_{i+1}$ are $\botadjac$-connected as well.
Starting from \eqref{krok:con0} 
and applying, consecutively, \eqref{krok:con3} and the observation above 
to $U_i, W_i$, $i=1,\dots, k$,
by induction, we get that $U, W$ are $\badjac$-connected.
This proves that
\begin{efekt}\label{krok:con4}
  the relation $\badjac$ on $\symsuby_k$ is connected.
\end{efekt}

\subsection{Proofs of the results}\label{subsec:proofs}

Now we are able to complete the proofs of our theorems from Sec.~\ref{sec:notions}
by gathering together the facts proved above.
The reasoning is more or less typical for Chow type theorems and its crucial
step consists in proving that the adjacency structure $\struct{\symsuby_{k-1},\badjac}$ 
(as well as $\struct{\symsuby_{k+1},\badjac}$) is definable in 
$\struct{\symsuby_k,\badjac}$.


\def\baba{{\goth B}}
Let $a_1,a_2$ be points of \fixprojr \ such that $a_1\nadjac a_2$.
By known properties of symplectic polar spaces and \ref{fct:copol}\eqref{copol:1}, 
the set 
  $\bigl\{ p\colon (\forall q)\;[\, a_1,a_2\nadjac q\implies q \nadjac p \,] \bigr\}$
is the projective line through $a_1,a_2$. 
Thus the class of lines of \fixproj \ is definable in \fixprojr. 
Finally, 
\begin{efekt}\label{krok:start}
  the metric projective geometry $(\fixproj,\perp)$ is definable
  in \fixprojr.
\end{efekt}
\par
Fix $k$ with $1 < k < n-1$.
Recall that, due to \eqref{krok:4} and \eqref{krok:8}, 
\begin{efekt}\label{krok:30}
   the two structures
     $\fixprojr_k := \PencSpace(\symsuby,k)$ and 
     $\baba_k := \struct{\symsuby_k,\badjac}$ 
   are mutually definable and consequently $\Aut(\fixprojr_k) = \Aut(\baba_k)$.
\end{efekt}
Let $\stars_k$ be the family of all the stars and $\tops_k$
be the family of all the tops in $\fixprojr_k$.
By \eqref{krok:3}, \eqref{krok:4}, and \eqref{krok:7},
in both cases $2 \dzieli k$ and $2 \ndzieli k$ the class
  $\stars_k \cup \tops_k$
is definable in  $\baba_k$ 
and thus it remains invariant under automorphisms of $\baba_k$.

\begin{lem}\label{lem:mainbis}
  Assume that stars and tops in $\PencSpace(\symsuby,k)$ are distinguishable.
  \begin{sentences}
  \item\label{lem:mainbis:a}
    The families $\symsuby_{k-1}$ and $\symsuby_{k+1}$ are definable in $\baba_k$.
	
  \item\label{lem:mainbis:c}
	If\/ $k > 1$ then  $\PencSpace(\symsuby,k-1)$ 
	can be defined in terms of the binary adjacency relation $\badjac$
	on $\symsuby_k$ and consequently, it can be defined in terms of\/ $\PencSpace(\symsuby,k)$.
	If\/ $k < n-1$ then $\PencSpace(\symsuby,k+1)$
	can be defined in terms of the binary adjacency relation $\badjac$
	on $\symsuby_k$ so, it can be defined in terms of\/ $\PencSpace(\symsuby,k)$.  

  \item\label{lem:mainbis:b}
    Each automorphism $F$ of\/ $\fixprojr_k$ determines an automorphism
      $F^+$ of\/ $\fixprojr_{k+1}$
	and an automorphism 
      $F^-$ of\/ $\fixprojr_{k-1}$ 
	such that\/
      $U\in\starof(H)$ iff\/ $F(U)\in\starof(F^-(H))$ 
	and\/
      $U\in\topof(B)$ iff\/ $F(U)\in\topof(F^+(B))$ 
	for all 
      $H\in\symsuby_{k-1}$, $U\in\symsuby_k$, and\/ $B\in\symsuby_{k+1}$.
  \end{sentences}
\end{lem}

\begin{proof}
  \eqref{lem:mainbis:a}
  It is enough to see that we can identify
  $H\in\symsuby_{k-1}$ with $\starof(H)$ and 
  $B\in\symsuby_{k+1}$ with $\topof(B)$. 
  
  \eqref{lem:mainbis:c}
  Let 
    $H_1,H_2\in\symsuby_{k-1}$,  and $B_1,B_2\in\symsuby_{k+1}$.
  Clearly, 
    $H_1 \topadjac H_2$ iff $\starof(H_1)\cap\starof(H_2)\neq\emptyset$,
  and dually
    $B_1 \botadjac B_2$ iff $\topof(B_1)\cap\topof(B_2)\neq\emptyset$.
  By \eqref{lem:mainbis:a} the maps $\starof(.)$ and $\topof(.)$ are definable, so
  up to them, 
  the relations: $\topadjac$ on $\symsuby_{k-1}$ and $\botadjac$ on 
  $\symsuby_{k+1}$ are definable in $\fixprojr_k$.
  This together with \eqref{krok:30} when $2\dzieli(k-1)$, and together with
  \eqref{krok:updown} when $2\ndzieli(k-1)$, proves our statement.

  \eqref{lem:mainbis:b}
  By  \eqref{lem:mainbis:a} we have bijections $F^+$ of\/ $\symsuby_{k+1}$
  and $F^-$ of\/ $\symsuby_{k-1}$ such that $\topof(F^+(B)) = F(\topof(B))$ 
  and $\starof(F^-(H)) = F(\starof(H))$. Now by \eqref{lem:mainbis:c} these
  maps are automorphisms as required.
\end{proof}

\begin{proof}[Proof of Proposition \ref{prop:main}]
  The proposition follows directly from Lemma~\ref{lem:mainbis}\eqref{lem:mainbis:c}, 
  \eqref{krok:3}, and \eqref{krok:9}.
\end{proof}

\begin{proof}[Proof of Theorem \ref{thm:main}]
  Assume that $k\neq n-k$. Then by \eqref{krok:3}, \eqref{krok:9} stars and tops
  are definable and distinguishable in $\fixprojr_k$. Hence, we can define
  $\fixprojr_{k-1}$ in $\fixprojr_k$ by Lemma~\ref{lem:mainbis}\eqref{lem:mainbis:c}.
  A top of $\fixprojr_{k-1}$ has form 
  $\{ H\in\symsuby_{k-1}\colon U\in\starof(H) \}$ for a point $U$ of $\fixprojr_k$, 
  so these tops are determined by the points of $\fixprojr_k$. 
  Analogously, the stars of $\fixprojr_{k+1}$ are determined by the
  points of $\fixprojr_k$.
  Therefore, the stars and the tops of $\fixprojr_{k-1}$  (of
  $\fixprojr_{k+1}$, resp.) can be distinguished.
  Now, from Lemma~\ref{lem:mainbis}\eqref{lem:mainbis:c} by induction on $k$ we infer 
  that $\fixprojr_1$, which is $\fixprojr$ up to an isomorphism, can be defined 
  in $\fixprojr_k$.  In view of \eqref{krok:start} we are through.
\end{proof}

\begin{proof}[Proof of Corollary \ref{cor:main}]
  Let $F\in\Aut(\baba_k)$. In view of \eqref{krok:30} it suffices to show that
  $F$ is induced by a collineation of $\fixproj$ preserving $\perp$ and acting 
  on $\symsuby_k$.
  
  In case $k\neq n - k$ the proof runs by induction on $k$ via Lemma~\ref{lem:mainbis}\eqref{lem:mainbis:b}.
  So, assume that $k = m$.
  Note that if 
	${\cal X}',{\cal X}''\in\stars_k$, $|{\cal X}'\cap{\cal X}''|=1$,
  and 
	$F({\cal X}')\in\stars_k$ 
  then 
	$F({\cal X}'')\in\stars_k$ as well.
  Indeed, if there were 
	$F({\cal X}'')\in\tops_k$ 
  then 
	$|F({\cal X}')\cap F({\cal X}'')|=1$ 
  would contradict \eqref{krok:peki1} and \eqref{krok:peki2}.
  Suppose that 
	$F({\cal X})\in\stars_k$ for some ${\cal X}=\starof(H)\in\stars_k$.
  By \eqref{krok:con4}, for each $H'\in\symsuby_{k-1}$ there is a sequence
	$H =H_0\topadjac\dots\topadjac H_t = H'$ 
  and then
	$F(\starof(H'))\in\stars_k$, 
  so $F$ preserves $\stars_k$.
  As in the proof of Lemma~\ref{lem:mainbis}\eqref{lem:mainbis:b}, we consider $F^+$ and $F^-$ to justify that $F$ is determined by 
  a collineation of \fixprojr.
  If 
	$F({\cal X}_0)\notin\stars_k$ for some ${\cal X}_0\in\stars_k$ 
  then
	$F({\cal X})\notin\stars_k$ for all ${\cal X}\in\stars_k$, 
  and thus
	$F({\cal X})\in\tops_k$ for all ${\cal X}\in\stars_k$.
  Let 
	$G = F\circ\varkappa$; 
  then $G\in\Aut(\baba_k)$ preserves the 
  families $\stars_k$ and $\tops_k$ and thus it is determined by a collineation.
  This completes the proof.
\end{proof}

\begin{proof}[Proof of Theorem~\ref{thm:updown}]
  We have 
    $\topadjac\;=\;\badjac$ for $k=1$ and 
    $\botadjac\;=\;\badjac$ for $k=n-1$. 
  Consequently, for $k\in\{1,n-1\}$, we are done by Theorem \ref{thm:main}.
  Let 
    $2\ndzieli k$ and $1 < k < n-1$.  
  By \eqref{krok:21} and \eqref{krok:22}, $\stars_k$ and $\tops_k$ can
  be distinguished in terms of $\botadjac$; 
  applying the duality $\varkappa$ we see that the stars and the tops are distinguishable in 
  terms of $\topadjac$ as well.
  As above, from \eqref{krok:updown}
  we get that $\fixprojr_{k-1}$ can be defined both in terms of $\botadjac$ 
  and in terms of $\topadjac$ on $\symsuby_k$.
  Finally, 
  this observation together with Theorem~\ref{thm:main}
  and Corollary~\ref{cor:main} completes the proof.
\end{proof}

%
\section{Comments}

We do not pay special attention to geometry of the Grassmannians introduced in the 
paper. Some comments which can be derived immediately from facts established in
Sec.~\ref{sec:proofs} are in order, though.
\begin{sentencesx}
\item\label{geo:ax1}
  If $2 \dzieli k$ then $\PencSpace(\symsuby,k)$ is a connected 
  partial linear $\Gamma$-space, whose 
  strong (or linear in other words)
  subspaces are affine spaces. 
  This resembles an affine polar space.
  And indeed, there are connections.
  The Grassmannian of regular lines in a projective $3$-space endowed with a 
  symplectic polarity is an affine polar space.
  In general, however, the axiom (3.1.iii) of \cite{cohenshult} fails
  here (and only this one from the list (3.1) fails).
  Moreover, our Grassmannians have maximal strong subspaces of two 
  distinct dimensions allowed, while maximal strong subspaces of a polar space
  all have the same dimension.
  And here, each line is the intersection of exactly two maximal strong subspaces
  which contain it.
\item\label{geo:ax2}
  If we consider the Grassmannian $\PencSpace(\symsuby,k)$ with $2 \dzieli k$
  as a part of the whole incidence geometry \eqref{def:symsuby} then we see
  that a plane in a strong subspace $\cal X$ 
  is determined either by an element of $\regsuby_{k\pm 2}$
  or by $X\in\Sub_{k\pm 2}({\field V})$ with $\rdim(X) =2$.
  Clearly, there is no way to distinguish these two types within the affine space $\cal X$.
  Let us consider an ``{\em affine geometry}" of the form
  \begin{ctext}
    ${\goth A}  = \langle \text{points of }{\cal X},\text{lines of }{\cal X},\Pi \rangle$,
  \end{ctext}
  where $\Pi$ are the regular planes in the above meaning.
  For any $3$-subspace $\Gamma$ of $\goth A$ 
  which contains a plane in $\Pi$
  the ``dual" structure 
  \begin{ctext}
    $\bstruct{ \{\pi\in\Pi\colon \pi\subset \Gamma  \},
      \{ L\colon L\subset \Gamma,\; L\text{ is a line of } {\goth A} \}, \supset}$
  \end{ctext}
  is an affine $3$-space.
\item\label{geo:ax3}
  Let $2 \ndzieli k$. Then $\PencSpace(\symsuby,k)$ is a connected 
  partial linear space which satisfies the following variant of the $\Delta$-axiom
  (cf. \cite{copolar}, \cite{embcopol}):
  \begin{ctext}\em
    a point not on a line $L$ is collinear with \\
    none, exactly one, or all except one point on $L$.
  \end{ctext}
  Let $\cal M$ be the family of the maximal subspaces of $\PencSpace(\symsuby,k)$
  which are (up to an isomorphism) symplectic copolar spaces.
  The family $\cal M$ covers the point set of our Grassmannian in such a way that
  any two elements of $\cal M$ intersect in a point, in a line, or are disjoint, 
  each line has exactly two extensions to a subspace in $\cal M$,
  and each clique of collinearity is contained in an element of $\cal M$.
  Very nice characterizations of the geometry on elements of $\cal M$ can be found, e.g. in
  \cite{embcopol} and~\cite{hipsympl}.
\end{sentencesx}

\medskip
We conjecture that starting with the properties \eqref{geo:ax1}--\eqref{geo:ax3}
one can obtain characterizations of respective Grassmannians in the style of 
\cite{talin}.


\bigskip
\begin{small}
\noindent
Authors' address:
\\
Ma{\l}gorzata Pra{\.z}mowska,
Krzysztof Pra{\.z}mowski,
Mariusz {\.Z}ynel
\\
Institute of Mathematics, University of Bia{\l}ystok
\\
ul. Akademicka 2, 15-267 Bia{\l}ystok, Poland
\\
\verb+malgpraz@math.uwb.edu.pl+, \verb+krzypraz@math.uwb.edu.pl+,
\verb+mariusz@math.uwb.edu.pl+
\end{small}

\end{document}